\documentclass[12pt,reqno]{amsart}

\usepackage{amssymb,amsmath}
\usepackage{amsfonts}
\usepackage{amsthm}
\usepackage{mathrsfs}
\usepackage{graphicx}
\usepackage[T1]{fontenc}

\theoremstyle{plain} 
\newtheorem{theorem}{Theorem}

\theoremstyle{definition} 
\newtheorem{definition}[theorem]{Definition}
\newtheorem{remark}[theorem]{Remark}
\newtheorem{example}[theorem]{Example}
\newcommand{\R}{\ensuremath{\mathbb{R}}}
\newcommand{\Z}{\ensuremath{\mathbb{Z}}}
\newcommand{\T}{\ensuremath{\mathbb{T}}}
\newcommand{\N}{\ensuremath{\mathbb{N}}}
\newcommand{\C}{\ensuremath{\mathbb{C}}}
\newcommand{\p}{\ensuremath{\mathbb{P}}}

\numberwithin{equation}{section}
\numberwithin{theorem}{section}

\small\normalsize

\begin{document}

\title{Equilibrium Stability for a Continuous Time Scale with Discrete Uniform gaps}
\author[Anderson]{Douglas R. Anderson} 
\address{Department of Mathematics \\
         Concordia College \\
         Moorhead, MN 56562 USA \\  https://orcid.org/0000-0002-3069-2816}
\email{andersod@cord.edu}

\begin{abstract} 
We investigate the equilibrium (trivial solution) stability, also known as Lyapunov stability, of a certain first-order linear complex constant coefficient dynamic equation on the time scale $\p_{\alpha,\beta}$, which has continuous intervals of length $\alpha>0$ followed by discrete gaps of length $\beta>0$. In particular, we establish results in the case of this specific time scale, for coefficient values in the complex plane, including where the exponential function alternates in sign. In our analysis, we employ the Lambert $W$ function. For increasing gap size $\beta$ relative to $\alpha$, we prove that the complex constant coefficient undergoes a bifurcation in its parameter space. We establish interesting results for both the delta dynamic equation and the nabla dynamic equation. Lastly, we connect these results to those related to Hyers--Ulam stability of the same nabla equations.
\end{abstract}


\maketitle
\thispagestyle{empty}


\section{Introduction}

Hilger \cite{hilger} introduced time scales analysis to unify continuous and discrete analysis, with the further benefit of supplying a robust tool for studying hybrid continuous-discrete systems and systems with non-uniform step sizes, where a time scale is any closed subset of the real line $\R$. A natural question for a time-scales equation relates to its stability. 
In a landmark paper, P\"{o}tzsche, Siegmund, and Wirth \cite{psw} gave a spectral characterization of exponential stability for linear time-invariant
systems on time scales. Using the first initials of these authors' surnames, these stability regions have become known as PSW regions. The tool for such analysis that PSW utilized, however, can be difficult to use in practice. In a recent series of papers, Jackson and Davis \cite{jd1,jd2} and Jackson, Davis, and Poulsen \cite{jdp} have used an ergodic approach to characterizing the PSW region, which is related to the region of convergence for the time-scales Laplace transform integral. 


In this paper, a specific time scale $\T$ with fixed gap size that displays both continuous and discrete properties is studied and its PSW stability regions are explicitly determined. In particular, let $\T=\p_{\alpha,\beta}$ for continuous interval length $\alpha>0$ and discrete gap size $\beta>0$, namely
\[ \p_{\alpha,\beta} = \bigcup_{k=0}^{\infty}[k(\alpha+\beta),k(\alpha+\beta)+\alpha], \]
and consider the differential operator defined by
\[ x^{\Delta}(t) = \begin{cases} \frac{d}{dt}x(t) &\text{for}\; t\in[k(\alpha+\beta),k(\alpha+\beta)+\alpha) \\ \frac{x(t+\beta)-x(t)}{\beta} &\text{for}\; t=k(\alpha+\beta)+\alpha. \end{cases} \]
See Bohner and Peterson \cite[Examples 1.38--1.40]{bp} for an introduction to this time scale, which can model a transmission signal or a data burst broadcast over a short time period, and then repeated, or a system or device that operates continuously for a fixed time, shuts off, and then runs again; in biology, this time scale may model an organism that lives a fixed unit of time, followed by hibernation or dormancy, and then is active again, and so on.  

We will be investigating some stability questions for this time scale and this derivative operator, in the case of the trivial solution to a linear equation, defined below.


\begin{definition}[Equilibrium Stability]
Consider the dynamic equation given by
\begin{equation}\label{maineq}
 x^\Delta(t) = \lambda x(t), \quad \lambda\in\C\backslash\left\{-\frac{1}{\beta}\right\}, \quad t\in\p_{\alpha,\beta}.
\end{equation}
We say that the equilibrium (trivial) solution of \eqref{maineq} is stable on $\p_{\alpha,\beta}$ if and only if given $\varepsilon>0$, there exists a $\delta>0$ such that if $x$ is a solution of \eqref{maineq} with $|x(0)|=|x_0|<\delta$, then for all $t\in\p_{\alpha,\beta}$ we have $|x(t)|<\varepsilon$ on $\p_{\alpha,\beta}$. Additionally, we say the equilibrium solution of \eqref{maineq} is asymptotically stable on $\p_{\alpha,\beta}$ if and only if there exists $\delta>0$ such that if $|x_0|<\delta$, then $\displaystyle\lim_{t\rightarrow\infty} |x(t)|=0$. The equilibrium solution of \eqref{maineq} is unstable if it is not stable.
\end{definition}

In this work, we consider the time scale $\T=\p_{\alpha,\beta}$, and the time scale eigenvalue problem given in \eqref{maineq}. For $t\in\T$, we have the forward gap operator $\sigma$ defined by
\[ \sigma(t):=\begin{cases} t&\text{for}\; t\in[k(\alpha+\beta),\;k(\alpha+\beta)+\alpha), \\ t+\beta&\text{for}\; t=k(\alpha+\beta)+\alpha. \end{cases} \]
For $\lambda\in\C\backslash\left\{-\frac{1}{\beta}\right\}$, the exponential function $e_\lambda(t,0)$ is given by
\[
 e_{\lambda}(t,0)=(1+\beta\lambda)^{k}e^{\lambda(t-k\beta)}, \quad t\in\big[k(\alpha+\beta), k(\alpha+\beta)+\alpha\big], \quad k\in\N_{0},
\]
which can also be written as 
\begin{equation}\label{pabexp0} 
 e_{\lambda}(t,0)=\left[(1+\beta\lambda)e^{\alpha\lambda}\right]^{k}e^{\lambda j}, \quad t=k(\alpha+\beta)+j, \quad j\in[0,\alpha]. 
\end{equation}
Clearly, the exponential function in \eqref{pabexp0} is well defined for $\lambda\in\C\backslash\left\{-\frac{1}{\beta}\right\}$. Notice that
\begin{equation}\label{xform}
 x(t)=x_0e_{\lambda}(t,0), \quad t\in\T,
\end{equation}
is the general solution of \eqref{maineq}, for the exponential function $e_{\lambda}$ given in \eqref{pabexp0}.

Throughout the paper, we will need to employ the Lambert $W$ function, see Corless et al \cite{Wfunction}, which we denote by $W_z$, where $W_z$ satisfies $W_z(u)e^{W_z(u)}=u$, for every $z\in\Z$. For example, using \eqref{pabexp0} and $t=k(\alpha+\beta)$, we have
\[ e_{\lambda}(k(\alpha+\beta),0) = [(1+\beta\lambda)e^{\alpha\lambda}]^k. \]
We will assume throughout that $\lambda\ne -\frac{1}{\beta}$, to prevent the exponential function from vanishing. Moreover, we will see that other key values for $\lambda\in\R$ include when the base $(1+\beta\lambda)e^{\alpha\lambda}=\pm 1$. If $\lambda=0$, then $(1+\beta\lambda)e^{\alpha\lambda}=1$, but, for $\lambda\in\R$, we note here that for the branches $z=-1,0$ of the Lambert $W=W_z$ function,
\[ (1+\beta\lambda)e^{\alpha\lambda}=-1 \Longleftrightarrow \lambda = \frac{-1}{\beta}+\frac{1}{\alpha}W_0\left(-\frac{\alpha}{\beta}e^{\frac{\alpha}{\beta}}\right) \;\text{and}\; \beta\ge\frac{\alpha}{W_0(e^{-1})}\approx 3.59112\alpha, \] 
where $\beta>0$ is the gap size, and $W_0$ is the principal branch of the Lambert $W$ function. In particular, if $\beta=\frac{\alpha}{W_0(e^{-1})}$ and $\lambda=-\frac{1}{\alpha}-\frac{1}{\beta}\approx \frac{-1.27846}{\alpha}$, then $e_{\lambda}(k(\alpha+\beta),0)=(-1)^k$.


\section{Equilibrium Stability on $\p_{\alpha,\beta}$}

We now give our first new results, when the eigenvalue $\lambda$ in \eqref{maineq} is a real number; in a later section, we will consider the more general case of $\lambda\in\C$. Moreover, we will fix $\alpha>0$ and let the gap size $\beta>0$ range over all positive real numbers in relation to $\alpha$. Of course, one could also fix $\beta>0$ and let $\alpha>0$ vary, as well. For the sake of completeness, we will include the details of proofs for this specific time scale 
$\T=\p_{\alpha,\beta}$. 


\begin{theorem}[Delta equation]\label{reallambdathm}
Fix $\alpha>0$, and let $\lambda\in\R\backslash\left\{-\frac{1}{\beta}\right\}$. We have the following cases.
\begin{enumerate}
\item Consider a small gap size $\beta$ that satisfies $0<\beta<\frac{\alpha}{W_0(e^{-1})}$.
 \begin{enumerate}
   \item If $\lambda\in\left(-\infty,\; -\frac{1}{\beta}\right)\cup\left(-\frac{1}{\beta},0\right)$, then the trivial solution of \eqref{maineq} is asymptotically stable. 
   \item If $\lambda=0$, then \eqref{maineq} is stable. 
   \item If $\lambda\in\left(0,\infty\right)$, then \eqref{maineq} is unstable.
 \end{enumerate}
\item Consider a critical gap size $\beta$ that satisfies $\beta=\frac{\alpha}{W_0(e^{-1})}$. Then, $(1+\beta\lambda)e^{\alpha\lambda}=-1$ at $\lambda=-\frac{1}{\alpha}\left(1+W_0(e^{-1})\right)$, and we have the following subcases.
 \begin{enumerate}
   \item If $\lambda\in\left(-\infty,\; -\frac{1}{\alpha}\left(1+W_0(e^{-1})\right)\right)\cup\left(-\frac{1}{\alpha}\left(1+W_0(e^{-1})\right),\; \frac{W_0(e^{-1})}{-\alpha}\right)\cup\left(\frac{W_0(e^{-1})}{-\alpha},0\right)$, 
	       then the trivial solution of \eqref{maineq} is asymptotically stable. 
   \item If $\lambda=-\frac{1}{\alpha}\left(1+W_0(e^{-1})\right)$ or $\lambda=0$, then the trivial solution of \eqref{maineq} is stable.
   \item If $\lambda\in(0,\infty)$, then the trivial solution of \eqref{maineq} is unstable.
 \end{enumerate}
\item Consider a large gap size $\beta$ that satisfies $\beta>\frac{\alpha}{W_0(e^{-1})}$. Then, $(1+\beta\lambda)e^{\alpha\lambda}=-1$ at 
 $$ \lambda_{-1}:=-\frac{1}{\beta}+\frac{1}{\alpha}W_{-1}\left(-\frac{\alpha}{\beta}e^{\frac{\alpha}{\beta}}\right) \quad\text{and}\quad 
    \lambda_{0}:=-\frac{1}{\beta}+\frac{1}{\alpha}W_{0}\left(-\frac{\alpha}{\beta}e^{\frac{\alpha}{\beta}}\right), $$ 
		and we have the following subcases.
 \begin{enumerate}
   \item If $\lambda\in\left(-\infty,\; \lambda_{-1}\right)\cup\left(\lambda_{0},\; -\frac{1}{\beta}\right)\cup\left(-\frac{1}{\beta},0\right)$, then the trivial solution 
	       of \eqref{maineq} is asymptotically stable. 
   \item If $\lambda=\lambda_{-1}$, $\lambda=\lambda_{0}$, or $\lambda=0$, then the trivial solution of \eqref{maineq} is stable.
   \item If $\lambda\in\left(\lambda_{-1},\lambda_{0}\right)\cup(0,\infty)$, 
	       then the trivial solution of \eqref{maineq} is unstable.
 \end{enumerate}
\end{enumerate}
\end{theorem}

\begin{proof}
Case (i)(a). Suppose $\lambda<-\frac{1}{\beta}$. Since $0<\beta<\frac{\alpha}{W_0(e^{-1})}$, the base of the exponential function \eqref{pabexp0} satisfies $(1+\beta\lambda)e^{\alpha\lambda}\in(-1,0)$, and $e_{\lambda}(t,0)=(1+\beta\lambda)^ke^{\lambda(t-k\beta)}$ goes to zero. Since 
$$ x(t) = x_0e_{\lambda}(t,0) $$
is a well-defined solution of \eqref{maineq}, it follows that the trivial solution of \eqref{maineq} is asymptotically stable. Suppose $-\frac{1}{\beta}<\lambda<0$. Then, the base of the exponential function \eqref{pabexp0} satisfies $(1+\beta\lambda)e^{\alpha\lambda}\in(0,1)$, and $e_{\lambda}(t,0)=(1+\beta\lambda)^ke^{\lambda(t-k\beta)}$ goes to zero; it follows that the trivial solution of \eqref{maineq} is again asymptotically stable.

Case (i)(b). If $\lambda=0$, then $e_{\lambda}(t,0)=(1)^ke^{0}=1$ for all $t\in\p_{\alpha,\beta}$, making the trivial solution of \eqref{maineq} stable.

Case (i)(c). If $\lambda>0$, then $(1+\beta\lambda)e^{\alpha\lambda}>1$, making the trivial solution of \eqref{maineq} unstable.

Case (ii)(a). Let $\beta=\frac{\alpha}{W_0(e^{-1})}$, and assume 
$$\lambda\in\left(-\infty,\; -\frac{1}{\alpha}\left(1+W_0(e^{-1})\right)\right)\cup\left(-\frac{1}{\alpha}\left(1+W_0(e^{-1})\right),\; \frac{W_0(e^{-1})}{-\alpha}\right)\cup\left(\frac{W_0(e^{-1})}{-\alpha},0\right).$$
Then, we have $(1+\beta\lambda)e^{\alpha\lambda}\in(-1,0)$; see \cite[p 7]{and}. This implies, for any $k\in\N_{0}$ and for all $t\in\T$, that
\[ \lim_{t\rightarrow\infty}|e_\lambda(t,0)|=0. \]
If $\lambda\in\left(\frac{W_0(e^{-1})}{-\alpha},0\right)$, then $(1+\beta\lambda)e^{\alpha\lambda}\in(0,1)$, and the trivial solution is asymptotically stable. 

Case (ii)(b). If $\lambda=-\frac{1}{\alpha}\left(1+W_0(e^{-1})\right)$, then $(1+\beta\lambda)e^{\alpha\lambda}=-1$; if $\lambda=0$, then $(1+\beta\lambda)e^{\alpha\lambda}=1$.  In either case, the trivial solution of \eqref{maineq} stable.

Case (ii)(c) is the same as Case (i)(c).

Case (iii)(a). If $\lambda\in\left(-\infty,\; \lambda_{-1}\right)\cup\left(\lambda_{0},\; -\frac{1}{\beta}\right)$, the base of the exponential function satisfies $-1<(1+\beta\lambda)e^{\alpha\lambda}<0$. If $\lambda\in\left(-\frac{1}{\beta},0\right)$, the base of the exponential function satisfies $0<(1+\beta\lambda)e^{\alpha\lambda}<1$. Consequently, the trivial solution of \eqref{maineq} is asymptotically stable.

Case (iii)(b). If $\lambda=\lambda_{-1}$ or $\lambda=\lambda_{0}$, then $(1+\beta\lambda)e^{\alpha\lambda}=-1$; if $\lambda=0$, then $(1+\beta\lambda)e^{\alpha\lambda}=1$. It follows that the trivial solution of \eqref{maineq} is stable in these cases.

Case (iii)(c). Let the exponential function be given by \eqref{pabexp0}. For $\lambda\in\left(\lambda_{-1},\lambda_{0}\right)$, the base of the exponential function satisfies $(1+\beta\lambda)e^{\alpha\lambda}<-1$. For $\lambda\in\left(0,\infty\right)$, the base of the exponential function satisfies $(1+\beta\lambda)e^{\alpha\lambda}>1$. Thus, in either case, the trivial solution is unstable.  
\end{proof}


\begin{example}
Fix $\alpha>0$, and let $\beta=\frac{2\alpha}{W_0(e^{-1})}>\frac{\alpha}{W_0(e^{-1})}$. Then, $(1+\beta\lambda)e^{\alpha\lambda}=-1$ at 
 $$ \lambda_{-1}:=-\frac{1}{\beta}+\frac{1}{\alpha}W_{-1}\left(-\frac{\alpha}{\beta}e^{\frac{\alpha}{\beta}}\right)\approx\dfrac{3.03479}{-\alpha} \quad\text{and}\quad 
    \lambda_{0}:=-\frac{1}{\beta}+\frac{1}{\alpha}W_{0}\left(-\frac{\alpha}{\beta}e^{\frac{\alpha}{\beta}}\right)\approx\dfrac{0.333598}{-\alpha}, $$
and 
$$ -\frac{1}{\beta} \approx \frac{0.139232}{-\alpha}. $$
Then, the following cases hold by Theorem \ref{reallambdathm}(iii).
 \begin{enumerate}
   \item If $\lambda\in\left(-\infty,\; \dfrac{3.03479}{-\alpha}\right)\cup\left(\dfrac{0.333598}{-\alpha},\; \dfrac{0.139232}{-\alpha}\right)\cup\left(\dfrac{0.139232}{-\alpha},0\right)$, then the trivial solution of \eqref{maineq} is asymptotically stable. 
   \item If $\lambda=\lambda_{-1}$, $\lambda=\lambda_{0}$, or $\lambda=0$, then the trivial solution of \eqref{maineq} is stable.
   \item If $\lambda\in\left(\dfrac{3.03479}{-\alpha},\;\dfrac{0.333598}{-\alpha}\right)\cup(0,\infty)$, 
	       then the trivial solution of \eqref{maineq} is unstable.
 \end{enumerate}
\end{example}


\begin{remark}\label{remark16}
Let $\lambda\in\R\backslash\left\{-\frac{1}{\beta}\right\}$. Fix the gap size $\beta>\frac{\alpha}{W_0(e^{-1})}$, as above in Theorem \ref{reallambdathm} (iii), and let $\alpha$ tend to 0. Then, 
$$ \lim_{\alpha\rightarrow 0}\p_{\alpha,\beta} = \beta\Z. $$
For $\alpha=0$, we see that $(1+\beta\lambda)=-1$ at 
$$ \lim_{\alpha\rightarrow 0}\lambda_{0}=\lim_{\alpha\rightarrow 0}\left(-\frac{1}{\beta}+\frac{1}{\alpha}W_{0}\left(-\frac{\alpha}{\beta}e^{\frac{\alpha}{\beta}}\right)\right)=-\frac{2}{\beta}, $$
and that $\displaystyle\lim_{\alpha\rightarrow 0}\lambda_{-1}=-\infty$. With $\alpha=0$, we have the following subcases from Theorem \ref{reallambdathm}(iii).
\begin{enumerate}				
\item[(a)] If $\lambda\in\left(-\frac{2}{\beta},\; -\frac{1}{\beta}\right)\cup\left(-\frac{1}{\beta},0\right)$, then the trivial solution of \eqref{maineq} is asymptotically stable. 
\item[(b)] If $\lambda=-\frac{2}{\beta}$ or $\lambda=0$, then the trivial solution of \eqref{maineq} is stable.
\item[(c)] If $\lambda\in\left(-\infty,\; -\frac{2}{\beta}\right)\cup(0,\infty)$, then the trivial solution of \eqref{maineq} is unstable.
\end{enumerate}
\end{remark}


\begin{remark}
If one uses the nabla backward difference operator on $\p_{\alpha,\beta}$ instead of the Delta forward difference operator, then analogous results may be derived. The nabla differential/difference operator in the nabla case is defined by
\[ x^{\nabla}(t) = \begin{cases} \frac{d}{dt}x(t) &: t\in(k(\alpha+\beta),k(\alpha+\beta)+\alpha] \\ \frac{x(t)-x(t-\beta)}{\beta} &: t=k(\alpha+\beta), \end{cases} \]
and the nabla exponential function is given via
\begin{equation}\label{nablaexpo}
 \widehat{e}_\lambda(t,0) = \frac{e^{\lambda t}}{[(1-\beta\lambda)e^{\beta\lambda}]^k}, \qquad \lambda\in\C\backslash\left\{\frac{1}{\beta}\right\}.
\end{equation}
Thus, for the nabla dynamic equation
\begin{equation}\label{nabeq}
 x^\nabla(t) = \lambda x(t), \quad \lambda\in\C\backslash\left\{\frac{1}{\beta}\right\}, \quad t\in\p_{\alpha,\beta},
\end{equation}
we may compare Theorem \ref{reallambdathm} with the following theorem.
\end{remark}


\begin{theorem}[Nabla equation]\label{realnablathm}
Fix $\alpha>0$, and let $\lambda\in\R\backslash\left\{\frac{1}{\beta}\right\}$. We have the following cases.
\begin{enumerate}
 \item Consider a small gap size $\beta$ that satisfies $0<\beta<\frac{\alpha}{W_0(e^{-1})}$.
  \begin{enumerate}
    \item If $\lambda\in(-\infty,0)$, then the trivial solution of \eqref{nabeq} is asymptotically stable. 
    \item If $\lambda=0$, then the trivial solution of \eqref{nabeq} is stable. 
    \item If $\lambda\in\left(0,\frac{1}{\beta}\right)\cup\left(\frac{1}{\beta},\infty\right)$, then the trivial solution of \eqref{nabeq} is unstable.
  \end{enumerate}
\item Consider a critical gap size $\beta$ that satisfies $\beta=\frac{\alpha}{W_0(e^{-1})}$. Then, $(1-\beta\lambda)^{-1}e^{\alpha\lambda}=-1$ at $\widehat{\lambda}=\frac{1}{\alpha}\left(1+W_0(e^{-1})\right)$, and we have the following subcases.
\begin{enumerate}
    \item If $\lambda\in(-\infty,0)$, then the trivial solution of \eqref{nabeq} is asymptotically stable.  
    \item If $\lambda=0$ or $\lambda=\frac{1}{\alpha}\left(1+W_0(e^{-1})\right)$, then the trivial solution of \eqref{nabeq} is stable.
    \item If $\lambda\in\left(0,\frac{W_0(e^{-1})}{\alpha}\right)\cup\left(\frac{W_0(e^{-1})}{\alpha},\;\frac{1}{\alpha}\left(1+W_0(e^{-1})\right)\right)\cup\left(\frac{1}{\alpha}\left(1+W_0(e^{-1})\right),\;\infty\right)$, then the trivial solution of \eqref{nabeq} is unstable.
\end{enumerate}
\item Consider a large gap size $\beta$ that satisfies $\beta>\frac{\alpha}{W_0(e^{-1})}$. Then, $(1-\beta\lambda)^{-1}e^{\alpha\lambda}=-1$ at 
      $$ \widehat{\lambda}_{-1}:=\frac{1}{\beta}-\frac{1}{\alpha}W_{-1}\left(-\frac{\alpha}{\beta}e^{\frac{\alpha}{\beta}}\right) \quad\text{and}\quad   
			\widehat{\lambda}_{0}:=\frac{1}{\beta}-\frac{1}{\alpha}W_{0}\left(-\frac{\alpha}{\beta}e^{\frac{\alpha}{\beta}}\right), $$ 
			and we have the following subcases.
\begin{enumerate}
\item If $\lambda\in(-\infty,0)\cup\left(\widehat{\lambda}_{0},\widehat{\lambda}_{-1}\right)$, then the trivial solution of \eqref{nabeq} is asymptotically stable. 
\item If $\lambda=0$, $\lambda=\widehat{\lambda}_{0}$, or $\lambda=\widehat{\lambda}_{-1}$, then the trivial solution of \eqref{nabeq} is stable.
\item If $\lambda\in\left(0,\frac{1}{\beta}\right)\cup\left(\frac{1}{\beta},\; \widehat{\lambda}_{0}\right)\cup\left(\widehat{\lambda}_{-1},\; \infty\right)$, 
then the trivial solution of \eqref{nabeq} is unstable.
\end{enumerate}
\end{enumerate}
\end{theorem}


\begin{remark}\label{remark26}
This is the nabla version of Remark \ref{remark16}. Let $\lambda\in\R\backslash\left\{\frac{1}{\beta}\right\}$. Fix the gap size $\beta>\frac{\alpha}{W_0(e^{-1})}$, as above in Theorem \ref{realnablathm} (iii), and let 
$\alpha$ tend to 0. As mentioned in Remark \ref{remark16}, 
$$ \lim_{\alpha\rightarrow 0}\p_{\alpha,\beta} = \beta\Z. $$
For $\alpha=0$, we see that $(1-\beta\lambda)=-1$ at 
$$ \lim_{\alpha\rightarrow 0}\widehat{\lambda}_{0}=\lim_{\alpha\rightarrow 0}\left(\frac{1}{\beta}-\frac{1}{\alpha}W_{0}\left(-\frac{\alpha}{\beta}e^{\frac{\alpha}{\beta}}\right)\right)=\frac{2}{\beta}, $$
and that $\displaystyle\lim_{\alpha\rightarrow 0}\widehat{\lambda}_{-1}=\infty$. Thus, with $\alpha=0$, we have the following subcases from 
Theorem \ref{realnablathm} (iii).
\begin{enumerate}	
\item[(a)]  If $\lambda\in(-\infty,0)\cup\left(\frac{2}{\beta},\infty\right)$, then the trivial solution of \eqref{nabeq} is asymptotically stable. 
\item[(b)] If $\lambda=0$ or $\lambda=\frac{2}{\beta}$, then the trivial solution of \eqref{nabeq} is stable.
\item[(c)] If $\lambda\in\left(0,\frac{1}{\beta}\right)\cup\left(\frac{1}{\beta},\; \frac{2}{\beta}\right)$, 
\end{enumerate}
then the trivial solution of \eqref{nabeq} is unstable.			
\end{remark}


\section{Complex Eigenvalues and Equilibrium Stability}

We extend the eigenvalues under consideration in this section to $\lambda\in\C\backslash\left\{-\frac{1}{\beta}\right\}$ for continuous interval size $\alpha>0$ and discrete gap size $\beta>0$ on the time scale $\p_{\alpha,\beta}$. To further motivate our use of the Lambert $W$ function, consider the exponential function given in \eqref{pabexp0}. Set the base of the exponential function as follows, $(1+\beta\lambda)e^{\alpha\lambda}=\rho e^{i\phi}$, for $\rho >0$, $i=\sqrt{-1}$, and $\phi\in(-\pi,\pi]$. Let $w=\frac{\alpha}{\beta}+\alpha\lambda$. Then, using the various branches in the complex plane of the Lambert $W$ function determined by $z\in\Z$, 
we have
\begin{equation}\label{complexlam}
 \lambda = -\frac{1}{\beta}+\frac{1}{\alpha}W_z\left(\frac{\rho \alpha}{\beta}e^{\frac{\alpha}{\beta}+i\phi}\right), \quad \phi\in(-\pi,\pi], \quad \rho >0, \quad \beta>0,
\end{equation} 
for $\phi\in(-\pi,\pi]$, with a branch cut along the negative real axis, where the principal branch is denoted $W_0$. 


\begin{theorem}[Delta equation]\label{complexTHM}
Let $(1+\beta\lambda)e^{\alpha\lambda}=\rho e^{i\phi}$, for $\rho >0$ and $\phi\in(-\pi,\pi]$, so that $\lambda\in\C\backslash\left\{-\frac{1}{\beta}\right\}$ is given by \eqref{complexlam}, where we have used the Lambert $W$ function denoted by $W_z$ for any $z\in\Z$. Then, the following hold.
\begin{enumerate}
\item[(a)] If $0<\rho<1$, then the trivial solution of \eqref{maineq} is asymptotically stable.
\item[(b)] If $\rho =1$, then the trivial solution of \eqref{maineq} is stable.
\item[(c)] If $\rho>1$, then the trivial solution of \eqref{maineq} is unstable.
\end{enumerate}
\end{theorem}

\begin{proof}
Let $\lambda\in\C\backslash\left\{-\frac{1}{\beta}\right\}$ be given as in \eqref{complexlam}, and let the Lambert $W$ function be denoted $W_z$ for any $z\in\Z$. 

Part (a).
First, let $\rho\in(0,1)$, and $\lambda=-\frac{1}{\beta}+\frac{1}{\alpha}W_z\left(\frac{\rho \alpha}{\beta}e^{\frac{\alpha}{\beta}+i\phi}\right)$ for $\phi\in(-\pi,\pi]$. Given the exponential function as in \eqref{pabexp0}, for $t=k(\alpha+\beta)+j$ and $j\in[0,\alpha]$, we then have 
\begin{equation}\label{thm31eq1}
 e_{\lambda}(t,0)=\left[(1+\beta\lambda)e^{\alpha\lambda}\right]^{k}e^{\lambda j}=\left(\rho e^{i\phi}\right)^{k}e^{\lambda j}. 
\end{equation}
Since $c e_{\lambda}(t,0)$ is a general solution to \eqref{maineq} for any $c\in\C$, clearly $c e_{\lambda}(t,0)\rightarrow 0$ as $t\rightarrow\infty$, making the trivial solution of \eqref{maineq} asymptotically stable when $\rho\in(0,1)$.

Part (b). If $\rho=1$, then $\lambda=-\frac{1}{\beta}+\frac{1}{\alpha}W_z\left(\frac{\alpha}{\beta}e^{\frac{\alpha}{\beta}+i\phi}\right)$ for $\phi\in(-\pi,\pi]$ and for fixed $z\in\Z$. Let the exponential function be given by \eqref{pabexp0}. Then, for $t=k(\alpha+\beta)+j\in[k(\alpha+\beta),k(\alpha+\beta)+\alpha]$ and $j\in[0,\alpha]$, 
\begin{eqnarray*}
 e_{\lambda}(t,0) &=& \left[(1+\beta\lambda)e^{\alpha\lambda}\right]^ke^{j\lambda} = e^{j\lambda+ik\phi}.
\end{eqnarray*}
Note that, for all $j\in[0,\alpha]$ and $\phi\in(-\pi,\pi]$, and for any fixed $z\in\Z$, the real part of $\lambda$ satisfies $\operatorname{Re}(\lambda)\le 0$, so that
\[ |e_{\lambda}(t,0)| = e^{j\operatorname{Re}(\lambda)}\in\left[e^{\alpha\operatorname{Re}(\lambda)},\;1\right]. \]
Consequently, with $\rho=1$, the trivial solution of \eqref{maineq} is stable.

Part (c). Let $\rho>1$, that is, let $\lambda=-\frac{1}{\beta}+\frac{1}{\alpha}W_z\left(\frac{\rho \alpha}{\beta}e^{\frac{\alpha}{\beta}+i\phi}\right)$, for $\phi\in(-\pi,\pi]$ and $z\in\Z$. By \eqref{thm31eq1} and the fact that $\rho>1$, we have that the trivial solution of \eqref{maineq} is unstable, as $|e_{\lambda}(t,0)|\rightarrow\infty$ as $t\rightarrow\infty$.
\end{proof}


\begin{figure}\label{fig1}
  \centering
		\includegraphics[scale=0.5]{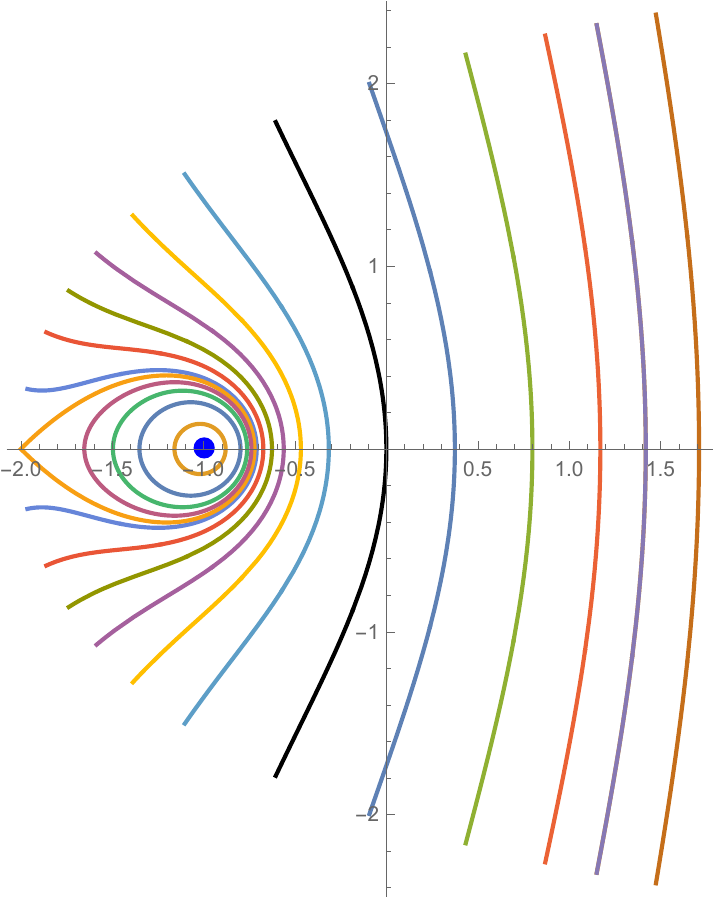}
		  \caption{Parametrized graphs of $\lambda=-1+W_0(\rho e^{1+i\phi})$ for various values of $\rho>0$, plotted for $\phi\in(-\pi,\pi]$, with $\alpha=\beta=1$. The black curve corresponds to $\rho=1$ and a stable manifold, where the trivial solution of \eqref{maineq} is stable. The curves to the left of the black curve are for $0<\rho <1$, and to the right for $\rho>1$. Note the homoclinic bifurcation at $\rho=\frac{-1}{e^2}$, including $\lambda=-2$. The blue dot is a hole in the complex plane for $\lambda=-1$.}
\end{figure}


\begin{remark}
Consider the nabla case with $\lambda\in\C\backslash\left\{\frac{1}{\beta}\right\}$ on the time scale $\p_{\alpha,\beta}$, for continuous interval size $\alpha>0$ and discrete gap size $\beta>0$. With the nabla exponential function given in \eqref{nablaexpo}, set the base of the exponential function as follows, $(1-\beta\lambda)^{-1}e^{\alpha\lambda}=\rho e^{i\phi}$, for $\rho>0$ and $\phi\in[0,2\pi)$. Again, for various branches of the Lambert $W$ function in the complex plane determined by $z\in\Z$, for $\phi\in[0,2\pi)$, we can solve for $\lambda$ to obtain
\begin{equation}\label{complexnab}
 \lambda = \frac{1}{\beta}-\frac{1}{\alpha}W_z\left(\frac{\alpha}{\rho\beta}e^{\frac{\alpha}{\beta}-i\phi}\right), \quad \phi\in[0,2\pi), \quad \rho >0, \quad \beta>0,
\end{equation} 
with a branch cut along the positive real axis, and principal branch denoted $W_0$. Then,
\begin{equation}\label{thm33eq1}
 \widehat{e}_{\lambda}(t,0)=\left(\frac{e^{\alpha\lambda}}{1-\beta\lambda}\right)^{k}e^{\lambda j}=\left(\rho e^{i\phi}\right)^{k}e^{\lambda j}, \quad j\in[0,\alpha],
\end{equation}
and we may now compare Theorem \ref{complexTHM} with the following theorem.
\end{remark}


\begin{theorem}[Nabla equation]\label{nabTHM}
Let $\lambda\in\C\backslash\left\{\frac{1}{\beta}\right\}$ have the form \eqref{complexnab}, where $W_z$ is the Lambert $W$ function for any $z\in\Z$. 
\begin{enumerate}
\item[(a)] If $0<\rho<1$, then the trivial solution of \eqref{nabeq} is asymptotically stable. 
\item[(b)] If $\rho=1$, then the trivial solution of \eqref{nabeq} is stable.
\item[(c)] If $\rho>1$, then the trivial solution of \eqref{nabeq} is unstable. 
\end{enumerate}
\end{theorem}


\section{Nabla Equation and Hyers--Ulam Stability}

The regions identified earlier in this work also arise when discussing another type of stability, that of Hyers--Ulam stability. Consequently, in this section we extend the analysis to the related but separate consideration of the Hyers--Ulam stability (HUS) of the nabla equation \eqref{nabeq}. Ulam \cite{ulam} initiated this type of discussion, followed by Hyers \cite{hyers} and Rassias \cite{rassias}. Similar results for the delta equation \eqref{maineq} are found in \cite{and}.

\begin{definition}[Hyers--Ulam Stability]
We say that \eqref{nabeq} has Hyers--Ulam stability on $\T$ if and only if given an arbitrary $\varepsilon>0$ and any function $\phi:\T\rightarrow\C$ satisfying 
\begin{equation}\label{phiineq}
 |\phi^\nabla(t)-\lambda\phi(t)|\le\varepsilon,  \quad t\in\T, 
\end{equation}
then there exists a solution $x:\T\rightarrow\C$ of \eqref{nabeq} and a constant $K>0$ such that $|\phi(t)-x(t)|\le K\varepsilon$ for all $t\in\T$. In this case, such a constant $K$ is called an HUS constant for \eqref{nabeq} on $\T$.
\end{definition}

Let
\begin{equation}\label{Krealnab}
 \widehat{K}_{\R}=\frac{\beta\lambda-1+e^{\alpha\lambda}(1-2\beta\lambda)}{\lambda\left(\beta\lambda-1-e^{\alpha\lambda}\right)}.
\end{equation}
The following result appeared in \cite[Theorem 2.5]{and}. It is included here for completeness and easy reference.


\begin{theorem}[Nabla equation]\label{realnablaHUS}
Fix $\alpha>0$, and let $\lambda\in\R\backslash\left\{\frac{1}{\beta}\right\}$. Also, let $\widehat{K}_{\R}$ be given as in \eqref{Krealnab}. We have the following cases.
\begin{enumerate}
 \item Suppose $0<\beta<\frac{\alpha}{W_0(e^{-1})}$.
  \begin{enumerate}
    \item If $\lambda\in(-\infty,0)\cup\left(0,\frac{1}{\beta}\right)$, then \eqref{nabeq} is Hyers--Ulam stable, with best HUS constant 
		      $K=\frac{1}{|\lambda|}$. 
    \item If $\lambda=0$, then \eqref{nabeq} is not Hyers--Ulam stable. 
    \item If $\lambda\in\left(\frac{1}{\beta},\infty\right)$, then \eqref{nabeq} is HUS, with best HUS constant $K=\widehat{K}_{\R}$.
  \end{enumerate}
\item Suppose $\beta=\frac{\alpha}{W_0(e^{-1})}$. Then, $(1-\beta\lambda)^{-1}e^{\alpha\lambda}=-1$ at $\lambda=\frac{1}{\alpha}\left(1+W_0(e^{-1})\right)$, and we have the following subcases.
\begin{enumerate}
    \item If $\lambda\in(-\infty,0)\cup\left(0,\frac{W_0(e^{-1})}{\alpha}\right)$, then \eqref{nabeq} is HUS, with best HUS constant 
          $K=\frac{1}{|\lambda|}$. 
    \item If $\lambda=0$ or $\lambda=\frac{1}{\alpha}\left(1+W_0(e^{-1})\right)$, then \eqref{nabeq} is not HUS.
    \item If $\lambda\in\left(\frac{W_0(e^{-1})}{\alpha},\;\frac{1}{\alpha}\left(1+W_0(e^{-1})\right)\right)\cup\left(\frac{1}{\alpha}\left(1+W_0(e^{-1})\right),\;\infty\right)$, then \eqref{nabeq} is HUS, with best HUS constant $K=\widehat{K}_{\R}$ as in \eqref{Krealnab}.
\end{enumerate}
\item Suppose $\beta>\frac{\alpha}{W_0(e^{-1})}$. Then, $(1-\beta\lambda)^{-1}e^{\alpha\lambda}=-1$ at 
      $$ \widehat{\lambda}_{-1}:=\frac{1}{\beta}-\frac{1}{\alpha}W_{-1}\left(-\frac{\alpha}{\beta}e^{\frac{\alpha}{\beta}}\right) \quad\text{and}\quad   
			\widehat{\lambda}_{0}:=\frac{1}{\beta}-\frac{1}{\alpha}W_{0}\left(-\frac{\alpha}{\beta}e^{\frac{\alpha}{\beta}}\right), $$ 
			and we have the following subcases.
\begin{enumerate}
\item If $\lambda\in(-\infty,0)\cup\left(0,\frac{1}{\beta}\right)$, then \eqref{nabeq} is HUS, with best HUS constant $K=\frac{1}{|\lambda|}$. 
\item If $\lambda=0$, $\lambda=\widehat{\lambda}_{0}$, or $\lambda=\widehat{\lambda}_{-1}$, then \eqref{nabeq} is not HUS.
\item If $\lambda\in\left(\frac{1}{\beta},\; \widehat{\lambda}_{0}\right)\cup\left(\widehat{\lambda}_{0},\widehat{\lambda}_{-1}\right)\cup\left(\widehat{\lambda}_{-1},\; \infty\right)$, then \eqref{nabeq} is HUS, with best HUS constant $K=\left|\widehat{K}_{\R}\right|$ as in \eqref{Krealnab}.
\end{enumerate}
\end{enumerate}
\end{theorem}

Next, we extend the considered values of the eigenvalue to $\lambda\in\C\backslash\left\{\frac{1}{\beta}\right\}$ on the time scale $\p_{\alpha,\beta}$ for continuous interval size $\alpha>0$ and discrete jump size $\beta>0$. As in the previous section, consider the exponential function given in 
\eqref{thm33eq1}. Again, set the base of the exponential function as follows, $(1-\beta\lambda)^{-1}e^{\alpha\lambda}=\rho e^{i\theta}$, for $\rho>0$, $i=\sqrt{-1}$, and $\theta\in[0,2\pi)$. Let $w=\frac{\alpha}{\beta}-\alpha\lambda$. Then, $\lambda$ is given by \eqref{complexnab}, for various branches of the Lambert $W$ function in the complex plane determined by $z\in\Z$, for $\theta\in[0,2\pi)$, with a branch cut along the negative real axis, and principal branch $W_0$. 

The following theorem is an improvement on \cite[Theorem 3.6]{and}.


\begin{theorem}\label{complexTHMHUS}
Let $\lambda\in\C\backslash\left\{\frac{1}{\beta}\right\}$ have the form \eqref{complexnab}, and let $W_z$ be the Lambert $W$ function for any $z\in\Z$. 
\begin{enumerate}
\item If $\rho=1$, then \eqref{nabeq} is not Hyers--Ulam stable.
\item If $\rho>1$, then \eqref{nabeq} is Hyers--Ulam stable, with HUS constant at most 
\begin{equation}\label{KRbigger1}
 K_{\rho>1}:=\max_{j\in[0,\alpha]}\left(\frac{\rho-1+e^{j\operatorname{Re}(\lambda)}+Re^{(j-\alpha)\operatorname{Re}(\lambda)}\left(\beta\operatorname{Re}(\lambda)-1\right)}{(\rho-1)\operatorname{Re}(\lambda)}\right), 
\end{equation} 
or $K=\frac{\rho(\beta+\alpha)}{\rho-1}$ if $\operatorname{Re}(\lambda)=0$. 
\item If $0<\rho<1$, then \eqref{nabeq} is Hyers--Ulam stable, with HUS constant at most  
\begin{equation}\label{Krsmaller1}
 K_{0<\rho<1}:= \max_{j\in[0,\alpha]}\left(\frac{\rho-1+e^{j\operatorname{Re}(\lambda)}+Re^{(j-\alpha)\operatorname{Re}(\lambda)}\left(\beta\operatorname{Re}(\lambda)-1\right)}{(1-\rho)\operatorname{Re}(\lambda)}\right). 
\end{equation}
\end{enumerate}
\end{theorem}

\begin{proof}
Let $\lambda\in\C\backslash\left\{\frac{1}{\beta}\right\}$ have the form \eqref{complexnab}, and let $W_z$ be the Lambert $W$ function for any $z\in\Z$. 

Case (i). If $\rho=1$, then $\lambda=\frac{1}{\beta}-\frac{1}{\alpha}W_z\left(\frac{\alpha}{\beta}e^{\frac{\alpha}{\beta}-i\theta}\right)$ for $\theta\in[0,2\pi)$ and for fixed $z\in\Z$. Let the exponential function be given by \eqref{pabexp0}. Then, for $t=k(\alpha+\beta)+j\in[k(\alpha+\beta),k(\alpha+\beta)+\alpha]$ and $j\in[0,\alpha]$, 
\begin{eqnarray*}
 \widehat{e}_{\lambda}(t,0) &=& (1-\beta\lambda)^{-k}e^{k\alpha\lambda} e^{j\lambda} = e^{j\lambda+ik\theta}.
\end{eqnarray*}
Note that, for all $j\in[0,\alpha]$ and $\theta\in[0,2\pi)$, and for any fixed $z\in\Z$, the real part of $\lambda$ satisfies $\operatorname{Re}(\lambda)\ge 0$, and
\[ |\widehat{e}_{\lambda}(t,0) | = e^{j\operatorname{Re}(\lambda)}\in\left[1,\;e^{\alpha\operatorname{Re}(\lambda)}\right]. \]
So, with $\widehat{e}_{\lambda}(t,0)=e^{j\lambda+ik\theta}$ for $t=k(\alpha+\beta)+j$, $j\in[0,\alpha]$, and $\theta\in[0,2\pi)$, set 
\[ \phi(t)=\frac{\varepsilon t\widehat{e}_{\lambda}(t,0)}{e^{\alpha\operatorname{Re}(\lambda)}}. \]
Then, we have
$$ |\phi^{\nabla}(t)-\lambda \phi(t)| = \left|\frac{\varepsilon\lambda t\widehat{e}_{\lambda}(t,0)}{e^{\alpha\operatorname{Re}(\lambda)}}+\frac{\varepsilon e^{\rho}_{\lambda}(t,0)}{e^{\alpha\operatorname{Re}(\lambda)}}-\frac{\varepsilon\lambda t\widehat{e}_{\lambda}(t,0)}{e^{\alpha\operatorname{Re}(\lambda)}}\right| = \frac{\varepsilon}{e^{\alpha\operatorname{Re}(\lambda)}}|e^{\rho}_{\lambda}(t,0)|\le \varepsilon $$
implies that $\phi$ satisfies \eqref{phiineq}, so that
$$ |\phi(t)-x(t)| = \left|\frac{\varepsilon t\widehat{e}_{\lambda}(t,0)}{e^{\alpha\operatorname{Re}(\lambda)}}-x_0\widehat{e}_{\lambda}(t,0)\right| = |\widehat{e}_{\lambda}(t,0)|\left|\frac{\varepsilon t}{e^{\alpha\operatorname{Re}(\lambda)}} - x_0\right|\ge \left|\frac{\varepsilon t}{e^{\alpha\operatorname{Re}(\lambda)}} - x_0\right| \rightarrow\infty $$
for any possible initial condition $x_0$, meaning \eqref{nabeq} is not HUS for $\rho=1$, that is when $\lambda=\frac{1}{\beta}-\frac{1}{\alpha}W_z\left(\frac{\alpha}{\beta}e^{\frac{\alpha}{\beta}-i\theta}\right)$ for any $\theta\in[0,2\pi)$, $\beta>0$, and for any fixed $z\in\Z$.

Case (ii). Let $\rho>1$, that is, let $\lambda = \frac{1}{\beta}-\frac{1}{\alpha}W_z\left(\frac{\alpha}{\beta \rho}e^{\frac{\alpha}{\beta}-i\theta}\right)$, initially with $\operatorname{Re}(\lambda)\ne 0$, for $\theta\in[0,2\pi)$ and $z\in\Z$. Let the exponential function be given by \eqref{pabexp0}, and let $\phi$ satisfy \eqref{phiineq}. Then, $\phi$ has the form given by 
\begin{equation}\label{phieqform}
 \phi(t)=\phi_0\widehat{e}_{\lambda}(t,0)+\widehat{e}_{\lambda}(t,0) \int_0^t\frac{q(s)}{\widehat{e}_{\lambda}(\rho(s),0)}\nabla s, \quad |q(s)|\le\varepsilon\; \forall s\in\T,  
\end{equation}
and again,
$$ \int_0^{\infty} \frac{q(s)}{\widehat{e}_{\lambda}(\rho(s),0)}\nabla s $$
exists and is finite, as $|\widehat{e}_{\lambda}(t,0)|=\rho^ke^{j\operatorname{Re}(\lambda)}$ for $\rho>1$ and $t=k(\alpha+\beta)+j$, $j\in[0,\alpha]$. Note that 
$$ x(t) = x_0\widehat{e}_{\lambda}(t,0), \qquad x_0=\phi_0 + \int_0^{\infty} \frac{q(s)}{\widehat{e}_{\lambda}(\rho(s),0)}\nabla s $$
is a well-defined solution of \eqref{nabeq}.
Now, to integrate from $s=0$ to $s=t=k(\alpha+\beta)+j$ for some $k\in\{0,1,2,\ldots\}$ and $j\in[0,\alpha]$, we see that there are $k$ continuous intervals and $k$ gaps to integrate over, plus the final partial interval (continuous), so that
\begin{eqnarray}
 \int_{0}^{t} \frac{\nabla s}{|\widehat{e}_{\lambda}(\rho(s),0)|} 
  &=& \sum_{m=0}^{k-1}\left(\int_{m(\alpha+\beta)}^{m(\alpha+\beta)+\alpha}\frac{ds}{|\widehat{e}_{\lambda}(s,0)|}\right)+\sum_{m=1}^{k}\left(\int_{m(\alpha+\beta)-\beta}^{m(\alpha+\beta)}\frac{\nabla s}{|\widehat{e}_{\lambda}(\rho(s),0)|}\right)+\int_{k(\alpha+\beta)}^{t}\frac{ds}{|\widehat{e}_{\lambda}(s,0)|} \nonumber\\
	&=& \sum_{m=0}^{k-1}\left(\int_{0}^{\alpha}\frac{e^{-j\operatorname{Re}(\lambda)}}{\rho^{m}}dj\right)+\sum_{m=0}^{k-1}\frac{\beta}{\rho^me^{\alpha\operatorname{Re}(\lambda)}} +\int_{0}^{j}\frac{e^{-\ell\operatorname{Re}(\lambda)}}{\rho^{k}}d\ell \nonumber\\
	&=& \sum_{m=0}^{k-1}\frac{-1}{\rho^m\operatorname{Re}(\lambda)}e^{-j\operatorname{Re}(\lambda)}\bigg|_{j=0}^{\alpha}+\sum_{m=0}^{k-1}\frac{\beta}{\rho^me^{\alpha\operatorname{Re}(\lambda)}}+\frac{-1}{\rho^k\operatorname{Re}(\lambda)}e^{-\ell\operatorname{Re}(\lambda)}\bigg|_{\ell=0}^{j} \nonumber\\
	&=& \frac{\rho(\rho^k-1)(1-e^{-\alpha\operatorname{Re}(\lambda)}))}{\rho^k(\rho-1)\operatorname{Re}(\lambda)} + \frac{e^{-\alpha\operatorname{Re}(\lambda)}\beta \rho(\rho^k-1)}{\rho^k(\rho-1)} + \frac{1-e^{-j\operatorname{Re}(\lambda)}}{\rho^k\operatorname{Re}(\lambda)} \label{int0tot}
\end{eqnarray}
and 
\begin{eqnarray*}
 \int_{0}^{\infty} \frac{\nabla s}{|\widehat{e}_{\lambda}(\rho(s),0)|} 
  = \lim_{t\rightarrow\infty}\int_{0}^{t} \frac{\nabla s}{|\widehat{e}_{\lambda}(\rho(s),0)|} 
	= \frac{\rho(1-e^{-\alpha\operatorname{Re}(\lambda)}))}{(\rho-1)\operatorname{Re}(\lambda)} + \frac{e^{-\alpha\operatorname{Re}(\lambda)}\beta \rho}{\rho-1} .
\end{eqnarray*}
Using these two integral values, we have
\begin{eqnarray*} 
 |\phi(t)-x(t)| &=& |\widehat{e}_{\lambda}(t,0)| \left|-\int_{t}^{\infty} \frac{q(s)}{\widehat{e}_{\lambda}(\rho(s),0)}\nabla s\right| \\
	&\le& \varepsilon |\widehat{e}_{\lambda}(t,0)| \int_{t}^{\infty} \frac{1}{|\widehat{e}_{\lambda}(\rho(s),0)|}\nabla s \\
	& = & \varepsilon |\widehat{e}_{\lambda}(t,0)| \left(\int_{0}^{\infty}-\int_{0}^{t} \right)\frac{1}{|\widehat{e}_{\lambda}(\rho(s),0)|}\nabla s  \\
	& = & \varepsilon \left(\frac{\rho-1+e^{j\operatorname{Re}(\lambda)}+\rho e^{(j-\alpha)\operatorname{Re}(\lambda)}\left(\beta\operatorname{Re}(\lambda)-1\right)}{(\rho-1)\operatorname{Re}(\lambda)}\right)
\end{eqnarray*}
for $j\in[0,\alpha]$, and for fixed $z\in\Z$, $\rho>1$, $\theta\in[0,2\pi)$ that determine $\lambda\in\C$ with $\operatorname{Re}(\lambda)\ne 0$. Set 
$K$ as in \eqref{KRbigger1}, that is,
$$ K_{\rho>1}:=\max_{j\in[0,\alpha]}\left(\frac{\rho-1+e^{j\operatorname{Re}(\lambda)}+\rho e^{(j-\alpha)\operatorname{Re}(\lambda)}\left(\beta\operatorname{Re}(\lambda)-1\right)}{(\rho-1)\operatorname{Re}(\lambda)}\right). $$
Therefore, \eqref{nabeq} has HUS for $\lambda = \frac{1}{\beta}-\frac{1}{\alpha}W_z\left(\frac{\alpha}{\beta \rho}e^{\frac{\alpha}{\beta}-i\theta}\right)$ with $\operatorname{Re}(\lambda)\ne 0$ and $\rho>1$, with HUS constant at most $K_{\rho>1}$. If $\operatorname{Re}(\lambda)=0$, then
\begin{eqnarray*} 
 \lim_{\operatorname{Re}(\lambda)\rightarrow 0} K_{\rho>1} &=& \lim_{\operatorname{Re}(\lambda)\rightarrow 0}\max_{j\in[0,\alpha]}\left(\frac{\rho-1+e^{j\operatorname{Re}(\lambda)}+\rho e^{(j-\alpha)\operatorname{Re}(\lambda)}\left(\beta\operatorname{Re}(\lambda)-1\right)}{(\rho-1)\operatorname{Re}(\lambda)}\right) \\
  &=& \max_{j\in[0,\alpha]}\frac{\rho(\beta+\alpha)+j(1-\rho)}{\rho-1} \\
	&=& \frac{\rho(\beta+\alpha)}{\rho-1},
\end{eqnarray*}
since $\rho>1$ in this case. As a result, \eqref{nabeq} is HUS with HUS constant at most $K_{\rho>1}=\frac{\rho(\beta+\alpha)}{\rho-1}$, for $\rho>1$ and $\operatorname{Re}(\lambda)=0$. In either instance, case (ii) holds.

Case (iii).
Finally, let $\lambda = \frac{1}{\beta}-\frac{1}{\alpha}W_z\left(\frac{\alpha}{\beta \rho}e^{\frac{\alpha}{\beta}-i\theta}\right)$ for $\theta\in[0,2\pi)$ and $\rho\in(0,1)$, let the exponential function be given by \eqref{pabexp0}, and let $\phi$ satisfy \eqref{phiineq}. Using $\rho\in(0,1)$ and \eqref{int0tot}, as well as $t=k(\alpha+\beta)+j$ for $j\in[0,\alpha]$, we can modify \eqref{int0tot} to get
\begin{eqnarray}
 \int_{0}^{t} \frac{\nabla s}{|\widehat{e}_{\lambda}(\rho(s),0)|} 
  &=& \frac{\rho(1-\rho^k)(1-e^{-\alpha\operatorname{Re}(\lambda)})}{\rho^k(1-\rho)\operatorname{Re}(\lambda)} + \frac{e^{-\alpha\operatorname{Re}(\lambda)}\beta \rho(1-\rho^k)}{\rho^k(1-\rho)} + \frac{1-e^{-j\operatorname{Re}(\lambda)}}{\rho^k\operatorname{Re}(\lambda)} \label{int0totsmallR}. 
\end{eqnarray}
If $\phi$ satisfies the perturbed equation \eqref{phiineq}, then $\phi$ is again given as in \eqref{phieqform}.
Let $x$ be a solution of \eqref{nabeq} with form \eqref{xform}, where $x_0=\phi_0-\varepsilon\left(\frac{e^{-\alpha\operatorname{Re}(\lambda)}\beta \rho}{1-\rho}+\frac{\rho(1-e^{-\alpha\operatorname{Re}(\lambda)})}{(1-\rho)\operatorname{Re}(\lambda)}\right)$; note that both fractions in the parentheses here are positive, due to $\rho\in(0,1)$ and $\operatorname{Re}(\lambda)<0$ in this case. Employing \eqref{int0totsmallR} with $t=k(\alpha+\beta)+j$, we see that
\begin{eqnarray*}
 |\phi(t)-x(t)| &=& |\widehat{e}_{\lambda}(t,0)|\left|\phi_0+\int_0^t \frac{q(s)}{\widehat{e}_{\lambda}(\rho(s),0)}\nabla s-\left(\phi_0-\varepsilon\left(\frac{e^{-\alpha\operatorname{Re}(\lambda)}\beta \rho}{1-\rho}+\frac{\rho(1-e^{-\alpha\operatorname{Re}(\lambda)})}{(1-\rho)\operatorname{Re}(\lambda)}\right)\right)\right| \\
 &=& |\widehat{e}_{\lambda}(t,0)|\left|\int_0^t \frac{q(s)}{\widehat{e}_{\lambda}(\rho(s),0)}\nabla s+\varepsilon\left(\frac{e^{-\alpha\operatorname{Re}(\lambda)}\beta \rho}{1-\rho}+\frac{\rho(1-e^{-\alpha\operatorname{Re}(\lambda)})}{(1-\rho)\operatorname{Re}(\lambda)}\right)\right| \\
 &\le& \varepsilon |\widehat{e}_{\lambda}(t,0)|\left(\int_0^t \frac{1}{|\widehat{e}_{\lambda}(\rho(s),0)|}\nabla s + \frac{e^{-\alpha\operatorname{Re}(\lambda)}\beta \rho}{1-\rho}+\frac{\rho(1-e^{-\alpha\operatorname{Re}(\lambda)})}{(1-\rho)\operatorname{Re}(\lambda)}\right) \\
 &=& \varepsilon \rho^ke^{j\operatorname{Re}(\lambda)}\left(\frac{\rho(1-e^{-\alpha\operatorname{Re}(\lambda)})}{\rho^k(1-\rho)\operatorname{Re}(\lambda)} + \frac{e^{-\alpha\operatorname{Re}(\lambda)}\beta \rho}{\rho^k(1-\rho)} + \frac{1-e^{-j\operatorname{Re}(\lambda)}}{\rho^k\operatorname{Re}(\lambda)}\right) \\
 &=& \varepsilon e^{j\operatorname{Re}(\lambda)} \left(\frac{\rho(1-e^{-\alpha\operatorname{Re}(\lambda)})}{(1-\rho)\operatorname{Re}(\lambda)} + \frac{e^{-\alpha\operatorname{Re}(\lambda)}\beta \rho}{1-\rho} + \frac{1-e^{-j\operatorname{Re}(\lambda)}}{\operatorname{Re}(\lambda)}\right) \\
 &\le& \varepsilon\left(\frac{\rho-1+e^{j\operatorname{Re}(\lambda)}+\rho e^{(j-\alpha)\operatorname{Re}(\lambda)}\left(\beta\operatorname{Re}(\lambda)-1\right)}{(1-\rho)\operatorname{Re}(\lambda)}\right)
\end{eqnarray*}
for $j\in[0,\alpha]$, as $\rho\in(0,1)$.
Therefore \eqref{nabeq} has HUS for $\lambda = \frac{1}{\beta}-\frac{1}{\alpha}W_z\left(\frac{\alpha}{\beta \rho}e^{\frac{\alpha}{\beta}-i\theta}\right)$ for $\rho\in(0,1)$, with HUS constant given by at most $K=K_{0<\rho<1}$ given in \eqref{Krsmaller1}.
This ends the proof.
\end{proof}

We end this section with another new result. 


\begin{theorem}
The HUS constants given in Theorem $\ref{realnablaHUS}$ and those given in Theorem $\ref{complexTHMHUS}$ are equivalent in the case where the eigenvalue $\lambda$ satisfies $\lambda\in\R$. 
\end{theorem}

\begin{proof}
Recall the polar substitution
\[ (1-\beta\lambda)^{-1}e^{\alpha\lambda} = \rho e^{i\theta}, \qquad \theta\in[0,2\pi). \]
Since the left-hand side is a real-valued expression for $\lambda\in\R$, we must take
\[ \theta = \begin{cases} 0 &\text{if}\quad 1-\beta\lambda>0, \\ \pi &\text{if}\quad 1-\beta\lambda<0. \end{cases} \]
In \eqref{KRbigger1}, take 
\[ \rho = \begin{cases} (1-\beta\lambda)^{-1}e^{\alpha\lambda} &\text{if}\quad 1-\beta\lambda>0, \\  (1-\beta\lambda)^{-1}e^{-\alpha\lambda} &\text{if}\quad 1-\beta\lambda<0. \end{cases} \]
Then, \eqref{KRbigger1} with $j=\alpha$ reduces to
\begin{align*} 
 K_{\rho>1} &=\max_{j\in[0,\alpha]}\left(\frac{\rho-1+e^{j\operatorname{Re}(\lambda)}+\rho e^{(j-\alpha)\operatorname{Re}(\lambda)}\left(\beta\operatorname{Re}(\lambda)-1\right)}{(\rho-1)\operatorname{Re}(\lambda)}\right) \\
 &= \begin{cases} \frac{1}{\lambda} &\text{if}\quad 1-\beta\lambda>0, \\ \frac{\beta\lambda-1+e^{\alpha\lambda}(1-2\beta\lambda)}{\lambda\left(\beta\lambda-1-e^{\alpha\lambda}\right)} &\text{if}\quad 1-\beta\lambda<0. \end{cases}
\end{align*}
Note the second constant is $\widehat{K}_{\rho}$ given in \eqref{Krealnab}. Similarly,
\[ K_{0<\rho<1} = \max_{j\in[0,\alpha]}\left(\frac{\rho-1+e^{j\operatorname{Re}(\lambda)}+\rho e^{(j-\alpha)\operatorname{Re}(\lambda)}\left(\beta\operatorname{Re}(\lambda)-1\right)}{(1-\rho)\operatorname{Re}(\lambda)}\right) = - \frac{1}{\lambda} \]
if $\lambda<0$. 
\end{proof}


\section*{Funding}
The author is supported by the AMS--Simons Research Enhancement Grants for Primarily Undergraduate Institution (PUI) Faculty (2025--2028), UID 265793.


\label{lastpage}
\end{document}